\documentclass[12pt]{amsart}

\usepackage{graphics, amsmath, amsfonts, amssymb, amscd, mathrsfs}

\newtheorem{theorem}{Theorem} 
\newtheorem{lemma}[theorem]{Lemma}
\newtheorem{corollary}[theorem]{Corollary}
\newtheorem{proposition}[theorem]{Proposition}

\theoremstyle{definition}
\newtheorem{definition}[theorem]{Definition}

\newcommand{\C}{\mathbb{C}}
\newcommand{\N}{\mathbb{N}}
\renewcommand{\ker}{\operatorname{Ker}}
\renewcommand{\Im}{\operatorname{Im}}

\title{Deformations of Oka manifolds}

\author{Finnur L\'arusson}
\address{School of Mathematical Sciences, University of Adelaide, Adelaide SA 5005, Australia.} 
\email{finnur.larusson@adelaide.edu.au}



\date{First version 23 December 2010.  Most recent changes 24 June 2011.}

\begin{document}

\begin{abstract}
We investigate the behaviour of the Oka property with respect to deformations of compact complex manifolds.  We show that in a family of compact complex manifolds, the set of Oka fibres corresponds to a $G_\delta$ subset of the base.  We give a necessary and sufficient condition for the limit fibre of a sequence of Oka fibres to be Oka in terms of a new uniform Oka property.  We show that if the fibres are tori, then the projection is an Oka map.  Finally, we consider holomorphic submersions with noncompact fibres.
\end{abstract}

\maketitle
\tableofcontents

\section{Introduction}

\noindent
The class of Oka manifolds has emerged from the modern theory of the Oka principle, initiated in 1989 in a seminal paper of Gromov \cite{Gromov}.  They were first formally defined by Forstneri\v c in 2009 in the wake of his result that some dozen possible definitions are all equivalent \cite{Forstneric1}.  The Oka property can be seen as an answer to the question: what should it mean for a complex manifold to be \lq\lq anti-hyperbolic\rq\rq?  For more background, see the survey \cite{Forstneric-Larusson}.  One of the many open problems in Oka theory is to clarify the place of Oka manifolds in the classification theory of compact complex manifolds, surfaces in particular.  To address this problem, we need to understand how the Oka property behaves with respect to deformations of compact complex manifolds.  In this paper, we take some first steps in this direction.

Let $\pi:X\to B$ be a family of compact complex manifolds, that is, a proper holomorphic submersion, and therefore a smooth fibre bundle, from a complex manifold $X$ onto a complex manifold $B$.  Write $X_t$ for the compact complex manifold $\pi^{-1}(t)$, $t\in B$.  We wish to say as much as possible about the structure of the set $O_\pi$ of $t\in B$ for which $X_t$ is Oka.

It is well known that the set of $t\in B$ for which $X_t$ is hyperbolic is open (\cite{Kobayashi}, Thm.\ 3.11.1).  Thus we might expect $O_\pi$ to be closed.  It is not clear that this is a reasonable conjecture.  We will prove two weaker results.  We show that $O_\pi$ is $G_\delta$ (Corollary \ref{c:gdelta}).  We also prove (Corollary \ref{c:passing-to-limit}) that if $t_n\to s$ in $B$ and $X_{t_n}$ is Oka for every $n\in\N$, then $X_s$ is Oka if (and in fact only if) the family $\{X_{t_n}:n\in\N\}$ is uniformly Oka in a sense introduced here (Definition \ref{d:uniformlyoka}).

As well as asking about Oka properties of the fibres of the projection $\pi$, we can ask about Oka properties of $\pi$ itself as a map.  It is an open question whether $\pi$ is an Oka map if all its fibres are Oka.  We show that the answer is affirmative if the fibres are tori (Theorem \ref{t:tori}).

Finally, we indicate how our results in Sections 2 and 3 can be extended to holomorphic submersions whose fibres are not necessarily compact.  We point out that every $G_\delta$ subset of a complex manifold $B$ can be realised as $O_\pi$ for some holomorphic submersion $\pi:X\to B$ with noncompact fibres.

\smallskip
\noindent
{\it Acknowledgement.}  I am grateful to Franc Forstneri\v c for helpful discussions.

\section{The convex approximation property for compact manifolds}

\noindent
We will formulate the Oka property of a complex manifold $Y$ as the convex approximation property:  $Y$ is Oka if and only if for every $k\geq 1$, every holomorphic map to $Y$ from a compact convex subset $K$ of $\C^k$ can be uniformly approximated on $K$ by holomorphic maps $\C^k\to Y$.  (By a holomorphic map on a compact set we always mean a holomorphic map on an open neighbourhood of the set.)

Let $Y$ be compact.  We fix a Hermitian metric $\omega$ on $Y$.  We use it to filter sets of holomorphic maps by normal families.  For our purposes, the choice of filtration seems immaterial.  To get a quantitative handle on the Oka property of $Y$, we introduce the following definition.

\begin{definition}
\label{d:sigma}
By a {\it quintuple} we shall mean a quintuple $(K, U, V, r, \epsilon)$, where $K$ is a nonempty compact subset of $\C^k$, $k\geq 1$, $U\subset V\Subset\C^k$ are open neighbourhoods of $K$, $r>0$, and $\epsilon>0$.  Note that $U$ and $V$ are assumed to be relatively compact in $\C^k$.  For every quintuple $(K, U, V, r, \epsilon)$, let
\[ \sigma(K, U, V, r, \epsilon)(Y)=\sup_{\substack{f:U\to Y \textrm{ hol.} \\ \lVert f^*\omega\rVert_U\leq r}} \inf_{\substack{g:V\to Y \textrm{ hol.} \\ d(f,g)<\epsilon \textrm{ on }K}} \lVert g^*\omega \rVert_V\in[0,\infty]. \]
\end{definition}

Here, $\lVert\cdot\rVert$ denotes the supremum norm with respect to the Euclidean metric on $\C^k$, and the distance $d(f,g)$ is with respect to $\omega$.  In the proof of Theorem \ref{t:semicontinuity1} below, it is important to have a weak inequality in $\lVert f^*\omega\rVert_U\leq r$ and a strict inequality in $d(f,g)<\epsilon$.  We take the infimum of the empty set to be $\infty$.

Clearly, $\sigma(K, U, V, r, \epsilon)(Y)$ increases as $r$ increases, $\epsilon$ decreases, $U$ shrinks, and $V$ expands.  Also, $\sigma(K, U, V, r, \epsilon)(Y)$ is finite if and only if there is $R>0$ such that every holomorphic map $f:U\to Y$ with $\lVert f^*\omega\rVert_U\leq r$ can be approximated to within $\epsilon$ on $K$ by a holomorphic map $g:V\to Y$ with $ \lVert g^*\omega \rVert_V\leq R$.  Since $Y$ is compact, whether $\sigma(K, U, V, r, \epsilon)(Y)$ is finite for all $r$ and $\epsilon$ with $K$, $U$, and $V$ fixed is independent of the choice of a Hermitian metric on $Y$.

\begin{proposition} 
\label{p:newoka}
The compact manifold $Y$ is Oka if and only if $\sigma(K, U, V, r, \epsilon)(Y)$ is finite for every quintuple $(K, U, V, r, \epsilon)$ such that $K$ is convex.
\end{proposition}

\begin{proof}  $\Leftarrow$  This is easy (and does not require compactness of $Y$).  

$\Rightarrow$  Suppose $\sigma(K, U, V, r, \epsilon)(Y)=\infty$ for some quintuple $(K, U, V, r, \epsilon)$ with $K$ convex.  This means that for every $n\in\N$, there is a holomorphic map $f_n:U\to Y$ with $\lVert f_n^*\omega\rVert_U\leq r$, such that every holomorphic map $g:V\to Y$ with $d(f_n,g)<\epsilon$ on $K$ (there may be none) has $\lVert g^*\omega \rVert_V>n$.  Since the family $\{f_n:n\in\N\}$ is equicontinuous, by passing to a subsequence, we may assume that $(f_n)$ converges locally uniformly on $U$ to a holomorphic map $f:U\to Y$.  Find $n_0$ such that $d(f,f_n)<\epsilon/2$ on $K$ for all $n\geq n_0$.  If $Y$ was Oka, we could find a holomorphic map $g:\C^n\to Y$ with $d(f,g)<\epsilon/2$ on $K$.  Then $d(f_n,g)<\epsilon$ on $K$ for $n\geq n_0$, and $\lVert g^*\omega \rVert_V\leq n$ for $n$ large enough, which is a contradiction.
\end{proof}

We now slightly modify Definition \ref{d:sigma}.

\begin{definition}
\label{d:tau}
For every quintuple $(K, U, V, r, \epsilon)$, let
\[ \tau(K, U, V, r, \epsilon)(Y)=\sup_{\substack{f:U\to Y \textrm{ hol.} \\ \lVert f^*\omega\rVert_U<r}} \inf_{\substack{g:V\to Y \textrm{ hol.} \\ d(f,g)\leq\epsilon \textrm{ on }K}} \lVert g^*\omega \rVert_V\in[0,\infty]. \]
\end{definition}

Note that the weak inequality in Definition \ref{d:sigma} has become strong and vice versa.  This will be important in the proof of Theorem \ref{t:semicontinuity2}.  Clearly,
\[ \tau(K, U, V, r, \epsilon)(Y)\leq \sigma(K, U, V, r, \epsilon)(Y). \]
The following result is proved in the same way as Proposition \ref{p:newoka}.

\begin{proposition} 
The compact manifold $Y$ is Oka if and only if $\tau(K, U, V, r, \epsilon)(Y)$ is finite for every quintuple $(K, U, V, r, \epsilon)$ such that $K$ is convex.
\end{proposition}

Now take a nonempty compact convex subset $K$ of $\C^k$, $k\geq 1$.  Choose a decreasing basis of open neighbourhoods $U_n$, $n\in\N$, of $K$, and an increasing sequence of open balls $V_n$, $n\in\N$, in $\C^k$ with radius going to infinity, such that $U_1\subset V_1$.  Let
\[ \sigma_n^K(Y)= \sigma(K, U_n, V_n, n, 1/n)(Y), \qquad \tau_n^K(Y)= \tau(K, U_n, V_n, n, 1/n)(Y).\]
Clearly, 
\[ \tau_1\leq\sigma_1\leq\tau_2\leq\sigma_2\leq\tau_3\leq\cdots. \]

\begin{definition}
\label{d:K-Oka}
Let $K$ be a compact convex subset of $\C^k$, $k\geq 1$.  We say that a complex manifold $Y$ is $K$-Oka if every holomorphic map $K\to Y$ can be uniformly approximated on $K$ by holomorphic maps $\C^k\to Y$.
\end{definition}

The following proposition is now immediate.

\begin{proposition}
\label{p:K-Oka}
Let $K$ be a nonempty compact convex subset of $\C^k$, $k\geq 1$.  The compact manifold $Y$ is $K$-Oka if and only if $\sigma_n^K(Y)$ is finite for all $n\in\N$.  Equivalently, $\tau_n^K(Y)$ is finite for all $n\in\N$.
\end{proposition}

We conclude this section by pointing out that the sequences $(\sigma_n^K(Y))$ and $(\tau_n^K(Y))$ go to infinity if $K$ has nonempty interior and $\dim Y\neq 0$.  Otherwise there is $c>0$ such that every holomorphic map $f_n:U_n\to Y$ with $\lVert f_n^*\omega\rVert_{U_n}\leq n$ can be approximated to within $1/n$ on $K$ by a holomorphic map $g_n:V_n\to Y$ with $\lVert g_n^*\omega\rVert_{V_n}\leq c$.  By equicontinuity it follows that every holomorphic map $K\to Y$ extends to $\C^k$ (more precisely, every holomorphic map from an open neighbourhood of $K$ to $Y$ is equal on $K$ to a holomorphic map $\C^k\to Y$).  If $K$ has nonempty interior and $\dim Y\neq 0$, this is absurd.

\section{The Oka property in a family of compact manifolds}

\noindent
Let $\pi:X\to B$ be a family of compact complex manifolds, that is, a proper holomorphic submersion from a complex manifold $X$ onto a complex manifold $B$.  Fix a base point $0$ in $B$.  Take a Hermitian metric $\omega$ on $X$.  Write $X_t$ for the compact complex manifold $\pi^{-1}(t)$, $t\in B$.  For a quintuple $(K, U, V, r, \epsilon)$, write $\sigma_{U,V}(t)$ or simply $\sigma(t)$ for $\sigma(K,U,V,r,\epsilon)(X_t)$, and $\tau_{U,V}(t)$ or simply $\tau(t)$ for $\tau(K,U,V,r,\epsilon)(X_t)$, defined using the metric $\omega|X_t$.

The following semicontinuity result is our first main theorem.

\begin{theorem}  
\label{t:semicontinuity1}
Let $(K, U, V_j, r, \epsilon)$, $j=1, 2$, be quintuples such that $V_1\Subset V_2$ and $V_2$ is Stein.  Then
\[\limsup_{t\to 0}\sigma_{U,V_1}(t) \leq \sigma_{U,V_2}(0). \]
\end{theorem}

The following two corollaries are consequences of the theorem and Proposition \ref{p:newoka}.

\begin{corollary}  
\label{c:gdelta}
The set of $t\in B$ for which $X_t$ is Oka is $G_\delta$.
\end{corollary}

\begin{proof}  
Now $\sigma$ is finite for all quintuples $(K, U, V, r, \epsilon)$ with $K$ convex if and only if $\sigma$ is finite for a suitable countable set of such quintuples.  Namely, we can take $r$ and $1/\epsilon$ to be integers, $V$ to be a ball of integer radius centred at the origin, and in between any compact convex $K$ and an open neighbourhood $U$ of $K$ we can fit the convex hull of a finite set of points with rational coordinates and the interior of a larger such hull.

Fix $K, U, r, \epsilon$ and take an increasing sequence $V_1\Subset V_2\Subset \cdots \Subset \C^k$ of Stein open neighbourhoods of $U$.  Write $\sigma_n$ for $\sigma_{U,V_n}$.  It suffices to show that $\bigcap\{\sigma_n<\infty\}$ is $G_\delta$.  This holds since by Theorem \ref{t:semicontinuity1}, $\{\sigma_n<\infty\}$ is a neighbourhood of $\{\sigma_{n+1}<\infty\}$.  In other words, for each $n\geq 1$, there is an open set $W_n$ with
\[\{\sigma_{n+1}<\infty\}\subset W_n\subset \{\sigma_n<\infty\},\] 
so $\bigcap\{\sigma_n<\infty\}=\bigcap W_n$.
\end{proof}

\begin{definition}
\label{d:uniformlyoka}
Let $A$ be a relatively compact subset of $B$.  The family $\{X_t:t\in A\}$ is {\it uniformly Oka} if $\sup\limits_{t\in A} \sigma(t)<\infty$ for every quintuple $(K, U, V, r, \epsilon)$ with $K$ convex.
\end{definition}

Requiring $A$ to be relatively compact makes the definition independent of the choice of $\omega$.  The next corollary, which is an immediate consequence of Theorem \ref{t:semicontinuity1}, suggests that the uniform Oka property is a reasonable notion.

\begin{corollary}  
\label{c:uniform-on-compacts}
If $A$ is a compact subset of $B$ and $X_t$ is Oka for all $t\in A$, then the family $\{X_t:t\in A\}$ is uniformly Oka.
\end{corollary}

To prove Theorem \ref{t:semicontinuity1} we need a lemma.  It is surely known, but for the reader's convenience we sketch a proof.

\begin{lemma}  
\label{l:perturb}
Let $M$ be a Stein manifold and $f:M\to X_0$ be holomorphic.  For every relatively compact open subset $\Omega$ of $M$, there is an open neighbourhood $W$ of $0$ in $B$ and a holomorphic map $F:\Omega\times W\to X$ such that $\pi\circ F=\mathrm{pr}_2$ and $F(\cdot,0)=f|\Omega$.
\end{lemma}

By an example of Brody and Green \cite{Brody-Green}, in general we cannot take $\Omega=M$.  (Their example also shows that the set of $t\in B$ for which $X_t$ is hyperbolic need not be closed.)

\begin{proof}
As a Stein submanifold of $M\times X$, the graph of $f$ has a Stein open neighbourhood $V$ in $M\times X$.  Let $p=\pi\circ\mathrm{pr}_2:V\to B$.  Since $p$ is a holomorphic submersion, $\ker p_*$ is a holomorphic subbundle of the holomorphic tangent bundle $TV$ of $V$.  Since $V$ is Stein, $TV$ splits into the direct sum of $\ker p_*$ and a holomorphic subbundle isomorphic to $p^* TB$.  Thus every holomorphic vector field on $B$ lifts by $p$ to one on $V$.

We may now proceed as in the usual proof of Ehresmann's fibration theorem.  We lift the vector fields $\partial/\partial z_j$ on $B$ (defined using local coordinates on a neighbourhood of $0$) to holomorphic vector fields on $V$, postcompose the map $m\mapsto (m,f(m))$ by the flows of the liftings, one after another, and then postcompose by $\mathrm{pr}_2:V\to X$.  The flows are defined for long enough to cover the same neighbourhood of $0$ in $B$ for all initial points in $\Omega$.
\end{proof}

\begin{proof}[Proof of Theorem  \ref{t:semicontinuity1}]
We argue by contradiction.  Suppose there is a sequence $(t_n)$ in $B$ converging to $0$ such that $\lim\limits_{n\to\infty}\sigma_{U,V_1}(t_n)$ exists (possibly equal to $\infty$) and 
\[ \lim\limits_{n\to\infty}\sigma_{U,V_1}(t_n) > \sigma_{U,V_2}(0).\]  
Write $X_n$ for $X_{t_n}$.  First suppose that $\lim\limits_{n\to\infty}\sigma_{U,V_1}(t_n)$ is finite; then we may assume that $\sigma_{U,V_1}(t_n)$ is finite for all $n\geq 1$, and we have
\[ \lim\limits_{n\to\infty}\sigma_{U,V_1}(t_n) > \sigma_{U,V_2}(0) +3\delta \]
for some $\delta>0$.  For each $n\geq 1$, find a holomorphic map $f_n:U\to X_n$ with $\lVert f_n^*\omega\rVert_U\leq r$ such that
\[ \inf_{\substack{g:V_1\to X_n \textrm{ hol.} \\ d(f_n,g)<\epsilon \textrm{ on }K}} \lVert g^*\omega \rVert_{V_1} > \sigma_{U,V_1}(t_n) - \delta,\]
so
\[ \inf_{\substack{g:V_1\to X_n \textrm{ hol.} \\ d(f_n,g)<\epsilon \textrm{ on }K}} \lVert g^*\omega \rVert_{V_1} > \lim\limits_{n\to\infty}\sigma_{U,V_1}(t_n) - 2\delta > \sigma_{U,V_2}(0)+\delta \]
for $n$ large enough.  By passing to a subsequence, we may assume that $(f_n)$ converges locally uniformly on $U$ to a holomorphic map $f:U\to X_0$.  Then $\lVert f^*\omega\rVert_U \leq r$.  There is a holomorphic map $h:V_2\to X_0$ such that $d(f,h)<\epsilon$ on $K$ and $\lVert h^*\omega\rVert_{V_2} < \sigma_{U,V_2}(0)+\delta$.  

Find an open set $V'$ in $\C^k$ with $V_1\Subset V'\Subset V_2$.  By Lemma \ref{l:perturb}, since $V_2$ is Stein, there is an open neighbourhood $W$ of $0$ in $B$ and a holomorphic map $H:V'\times W\to X$ such that $\pi\circ H=\mathrm{pr}_2$ and $H(\cdot,0)=h|V'$.   For $n$ sufficiently large, setting $g_n=H(\cdot,t_n)|V_1\to X_n$, we have $d(f_n,g_n)<\epsilon$ on $K$ and $\lVert g_n^*\omega\rVert_{V_1} < \sigma_{U,V_2}(0)+\delta$, which is a contradiction.

When $\lim\limits_{n\to\infty}\sigma_{U,V_1}(t_n)=\infty$, there are holomorphic maps $f_n:U\to X_n$ with $\lVert f_n^*\omega\rVert_U\leq r$ such that
\[ \inf_{\substack{g:V_1\to X_n \textrm{ hol.} \\ d(f_n,g)<\epsilon \textrm{ on }K}} \lVert g^*\omega \rVert_{V_1} > \sigma_{U,V_2}(0)+ 1 \]
for $n$ large enough, and we conclude as above.
\end{proof}

Our second main result is dual to Theorem \ref{t:semicontinuity1}.

\begin{theorem}  
\label{t:semicontinuity2}
Let $(K, U_j, V, r, \epsilon)$, $j=1, 2$, be quintuples such that $U_2\Subset U_1$ and $U_1$ is Stein.  Then
\[\tau_{U_1,V}(0)\leq\liminf_{t\to 0}\tau_{U_2,V}(t). \]
\end{theorem}

\begin{proof}
We argue by contradiction.  Suppose there is a sequence $(t_n)$ in $B$ converging to $0$ such that $\lim\limits_{n\to\infty}\tau_{U_2,V}(t_n)$ exists and 
\[ \lim\limits_{n\to\infty}\tau_{U_2,V}(t_n) < \tau_{U_1,V}(0),\]
so in particular $\lim\limits_{n\to\infty}\tau_{U_2,V}(t_n)$ is finite and we may assume that $\tau_{U_2,V}(t_n)$ is finite for all $n$.  Find a holomorphic map $f:U_1\to X_0$ such that $\lVert f^*\omega\rVert_{U_1} < r$ and
\[ \lim\limits_{n\to\infty}\tau_{U_2,V}(t_n)+\delta < \inf_{\substack{g:V\to X_0 \textrm{ hol.} \\ d(f,g)\leq\epsilon \textrm{ on }K}} \lVert g^*\omega \rVert_V \]
with $\delta>0$ (the infimum might be infinite).  

Find an open set $U'$ in $\C^k$ with $U_2\Subset U'\Subset U_1$.  By Lemma \ref{l:perturb}, since $U_1$ is Stein, there is an open neighbourhood $W$ of $0$ in $B$ and a holomorphic map $F:U'\times W\to X$ such that $\pi\circ F=\mathrm{pr}_2$ and $F(\cdot,0)=f|U'$.  We may assume that $t_n\in W$ for all $n$.  Let $f_n=F(\cdot,t_n)|U_2\to X_{t_n}$.  Then $\lVert f_n^*\omega\rVert_{U_2} < r$ for $n$ large enough.  For such $n$, there is a holomorphic map $g_n:V\to X_{t_n}$ with $d(f_n,g_n)\leq\epsilon$ on $K$ and $\lVert g_n^*\omega \rVert_{V}<\tau_{U_2,V}(t_n)+\delta/2$, so $\lVert g_n^*\omega \rVert_{V}<\lim\limits_{n\to\infty}\tau_{U_2,V}(t_n)+\delta$ for $n$ large enough.

By passing to a subsequence, we may assume that $(g_n)$ converges locally uniformly on $V$ to a holomorphic map $g:V\to X_0$.  Then $d(f,g)\leq\epsilon$ on $K$ and 
\[ \lVert g^*\omega \rVert_{V}\leq \lim\limits_{n\to\infty}\tau_{U_2,V}(t_n)+\delta, \]
which is a contradiction.
\end{proof}

The following result is an immediate consequence of Theorem \ref{t:semicontinuity2}, as it is clear that $\sigma$ can be replaced by $\tau$ in Definition \ref{d:uniformlyoka}.

\begin{corollary}  
\label{c:closure}
If $A$ is a relatively compact subset of $B$ and the family $\{X_t:t\in A\}$ is uniformly Oka, then the family $\{X_t:t\in \bar A\}$ is also uniformly Oka.  In particular, $X_t$ is Oka for every $t\in\bar A$.
\end{corollary}

The next result follows from Corollaries \ref{c:uniform-on-compacts} and \ref{c:closure}.

\begin{corollary}
\label{c:passing-to-limit}
Let $t_n\to 0$ in $B$ and suppose $X_{t_n}$ is Oka for all $n\in\N$.  Then $X_0$ is Oka if and only if the family $\{X_{t_n}:n\in\N\}$ is uniformly Oka. 
\end{corollary}

An examination of the proofs of Theorems \ref{t:semicontinuity1} and \ref{t:semicontinuity2} shows that for Corollary \ref{c:passing-to-limit} to hold, the Hermitian metrics on $X_{t_n}$ used to define what it means for the family $\{X_{t_n}:n\in\N\}$ to be uniformly Oka need not be the restrictions of a Hermitian metric on $X$ as assumed above.  It suffices that the metrics on $X_{t_n}$ converge to a Hermitian metric on the central fibre $X_0$ as $n\to\infty$.

It remains an open question whether $X_0$ must be Oka if $X_t$ is Oka for all $t\in B\setminus\{0\}$.  The following corollary of Theorems \ref{t:semicontinuity1} and \ref{t:semicontinuity2} explains how $X_0$ would fail to be Oka.  

\begin{corollary}
The central fibre $X_0$ is not Oka if and only if $\lim\limits_{t\to 0}\sigma_n^K(t)=\infty$ for some nonempty compact convex subset $K$ of $\C^k$ and every sufficiently large $n\in\N$.
\end{corollary}

\section{Families of tori}

\noindent
Let $\pi:X\to B$ be a family of compact complex manifolds.  It is an open question whether $\pi$ is an Oka map if all its fibres are Oka manifolds.

\begin{theorem}
\label{t:tori}
The projection of a family of complex tori is an Oka map.
\end{theorem}

\begin{proof}
Let $\mathscr B$ be the space of $n\times n$ complex matrices $T$ with $\det\Im T>0$.  Let $\mathscr X$ be the quotient of $\C^n\times\mathscr B$ by the free and properly discontinuous action of $\mathbb Z^{2n}$ given by the formula
\[ m\cdot (z,T) = (z+m\cdot \left[ \begin{array}{c} I \\ T \end{array} \right] , T), \]
with the induced projection $p:\mathscr X\to\mathscr B$.  It is well known that the family $p$ is complete and effective and contains every $n$-dimensional torus (\cite{Kodaira-Spencer}; see also \cite{Birkenhake-Lange}, \S 7.3).  In other words, $p$ is a minimal versal deformation of every $n$-dimensional torus.

Next observe that $p$ has a dominating fibre spray $\sigma$ defined on the trivial vector bundle $\mathscr X\times\C^n$ over $\mathscr X$ by the formula
\[ \sigma([z,T],w)=[z+w,T]. \]
Hence $p$ is elliptic and thus Oka.

Finally, let $\pi:X\to B$ be a family of $n$-dimensional tori.  Let $t\in B$.  By completeness of the family $p$, there is an open neighbourhood $U$ of $t$ in $B$ and a holomorphic map $f:U\to\mathscr B$ such that the restriction $\pi | \pi^{-1}(U)\to U$ is isomorphic to the pullback family $f^*p$.  Since pullbacks of Oka maps are Oka, we conclude that $\pi$ is Oka over a neighbourhood of each point in $B$.  Hence $p$ is Oka (\cite{Forstneric2}, Theorem 4.7).
\end{proof}

We remark that a holomorphic map with noncompact Oka fibres need not be Oka.  There is a smoothly trivial holomorphic submersion which is not Oka but all of whose fibres are isomorphic to $\C^*$ (\cite{Forstneric-Larusson}, Example 6.6).

\section{The noncompact case}

\noindent
The definitions and results in Sections 2 and 3 can be extended to the noncompact case.  First, let $Y$ be a complex manifold with a complete Hermitian metric $\omega$.  By a {\it sextuple} we mean a sextuple $(K,U,V,r,\epsilon,L)$, where $(K,U,V,r,\epsilon)$ is a quintuple as defined above, and $L$ is a relatively compact subset of $Y$.  Let
\[ \sigma(K, U, V, r, \epsilon, L)(Y)=\sup_{\substack{f:U\to Y \textrm{ hol.} \\ \lVert f^*\omega\rVert_U\leq r \\ f(U)\subset L}} \inf_{\substack{g:V\to Y \textrm{ hol.} \\ d(f,g)<\epsilon \textrm{ on }K }} \lVert g^*\omega \rVert_V. \]
We define $\tau(K, U, V, r, \epsilon, L)(Y)$ similarly, with the weak and the strong inequality signs interchanged.  As before, we can show that $Y$ is Oka if and only if $\sigma(K, U, V, r, \epsilon, L)(Y)$, or equivalently $\tau(K, U, V, r, \epsilon, L)(Y)$, is finite for every septuple $(K,U,V,r,\epsilon,L)$ such that $K$ is convex.

Next, let $\pi:X\to B$ be a holomorphic submersion from a complex manifold $X$ onto a complex manifold $B$.  Fix a base point $0$ in $B$.  Write $X_t=\pi^{-1}(t)$ for $t\in B$.  Choose a complete Hermitian metric $\omega$ on $X$.  Let $L$ be a subset of $X$ such that $L\cap X_t$ is relatively compact for every $t\in B$.  We define $\sigma(t)$ and $\tau(t)$ using sextuples $(K,U,V,r,\epsilon,L\cap X_t)$ and the metric $\omega|X_t$.  With minor modifications, the proofs of Theorems \ref{t:semicontinuity1} and \ref{t:semicontinuity2} show that 
\[\limsup_{t\to 0}\sigma_{U,V_1}(t) \leq \sigma_{U,V_2}(0) \]
if $V_1\Subset V_2$ and $L$ is closed, and
\[\tau_{U_1,V}(0)\leq\liminf_{t\to 0}\tau_{U_2,V}(t) \]
if $U_2\Subset U_1$ and $L$ is open.  Using an exhaustion of $X$ by compact subsets $L$, the proof of Corollary \ref{c:gdelta} yields the following generalisation.

\begin{theorem}  
\label{t:general-gdelta}
Let $\pi:X\to B$ be a holomorphic submersion.  The set of $t\in B$ for which $X_t$ is Oka is $G_\delta$.
\end{theorem}

Let $A$ be a relatively compact subset of $B$.  We generalise Definition \ref{d:uniformlyoka} and say that the family $\{X_t:t\in A\}$ is uniformly Oka if $\sup\limits_{t\in A} \sigma(t)<\infty$ for every sextuple $(K, U, V, r, \epsilon,L)$ with $K$ convex and $L$ relatively compact in $X$.  Clearly, an equivalent definition results from replacing $\sigma$ by $\tau$, or from requiring $L$ to be compact or to be open.  The definition appears to depend on the choice of $\omega$, but Corollaries \ref{c:uniform-on-compacts} and \ref{c:closure}, whose proofs still go through, show that in fact it does not.  We have the following generalisation of Corollary \ref{c:passing-to-limit}.

\begin{theorem}  Let $\pi:X\to B$ be a holomorphic submersion.  Let $t_n\to 0$ in $B$ and suppose $X_{t_n}$ is Oka for all $n\in\N$.  Then $X_0$ is Oka if and only if the family $\{X_{t_n}:n\in\N\}$ is uniformly Oka.
\end{theorem}

One might expect Theorem \ref{t:general-gdelta} to be far from optimal, due to the seemingly primitive approach of treating the Oka property as the conjunction of a countable number of unrelated properties.  However, for submersions that are not necessarily proper, it is easily seen that Theorem \ref{t:general-gdelta} is in fact sharp.

\begin{proposition}
Let $B$ be a complex manifold and $A\subset B$ be $G_\delta$.  There is a holomorphic submersion $\pi:X\to B$ such that the set of $t\in B$ for which $\pi^{-1}(t)$ is Oka is precisely $A$.
\end{proposition}

\begin{proof}
Write $B\setminus A$ as a countable union $\bigcup\limits_{n=1}^\infty F_n$ of closed sets (we may assume the union is infinite, as there is no harm in listing a set more than once).  The subset $E=\bigcup\limits_{n=1}^\infty F_n\times\{n,n+1\}$ of $B\times \C$ is closed.  The intersection $E\cap(\{t\}\times\C)$ contains at least two points if $t\in B \setminus A$, but is empty if $t\in A$.  Hence the projection $(B\times\C)\setminus E\to B$ has the required property.
\end{proof}

\end{document}